\documentclass[reqno]{amsart}
\usepackage{amsfonts,amssymb,graphicx,float}

\makeatletter                                                  %
\def\th@plain{\slshape}                                        %
\makeatother                                                   %

\newcommand{\oi}{[0,1]}
\newcommand{\Nbb}{\mathbb{N}}
\newcommand{\Zbb}{\mathbb{Z}}
\newcommand{\Qbb}{\mathbb{Q}}
\newcommand{\Rbb}{\mathbb{R}}
\newcommand{\Cbb}{\mathbb{C}}
\newcommand{\one}{{\rm 1\mskip-4mu l}}
\newcommand{\op}{\ordine_p}
\newcommand{\To}{$\,\Rightarrow\,$}

\newcommand{\labell}[1]{\label{#1}}

\newcommand{\newword}[1]{\textsl{#1}}
\newcommand{\vect}[3]{#1_#2,\ldots ,#1_#3}

\newcommand{\abs}[1]{\lvert#1\rvert}
\newcommand{\norm}[1]{\lVert#1\rVert}
\newcommand{\floor}[1]{\lfloor #1 \rfloor}
\newcommand{\set}[1]{\{ #1 \}}
\newcommand{\mobiusfrac}[2]{\mu\biggl(\frac{#1}{#2}\biggr)}

\DeclareMathSymbol{\upharpoonright}{\mathrel}{AMSa}{"16}

\DeclareMathSymbol{\nmid}{\mathrel}{AMSb}{"2D}

\DeclareMathOperator{\den}{den}
\DeclareMathOperator{\ordine}{o}
\DeclareMathOperator{\GL}{GL}
\DeclareMathOperator{\Lin}{Lin}
\DeclareMathOperator{\rad}{rad}
\DeclareMathOperator{\Mat}{Mat}
\DeclareMathOperator{\dom}{dom}

\theoremstyle{plain}
\newtheorem{theorem}{Theorem}[section]
\newtheorem{lemma}[theorem]{Lemma}

\newtheorem{corollary}[theorem]{Corollary}

\theoremstyle{definition}

\newtheorem{remark}[theorem]{Remark}

\begin{document}

\bibliographystyle{plain}

\sloppy

\title[Denominator-preserving maps]{Denominator-preserving maps}

\author[]{Giovanni Panti}
\address{Department of Mathematics\\
University of Udine\\
via delle Scienze 208\\
33100 Udine, Italy}

\begin{abstract}
Let $F$ be a continuous injective map from an open subset of $\Rbb^n$ to $\Rbb^n$. Assume that, for infinitely many $k\ge1$, $F$ induces a bijection between the rational points of denominator $k$ in the domain and those in the image (the denominator of $(a_1/b_1,\ldots,a_n/b_n)$ being the l.c.m.~of $\vect b1n$). Then $F$ preserves the Lebesgue measure.
\end{abstract}


\thanks{\emph{2010 Math.~Subj.~Class.}: 28D05; 11K06}

\maketitle

\section{Preliminaries and statement of the main results}
For every point $u$ in $\Qbb^n$ there exist uniquely determined relatively prime integers $\vect a1n,k$ such that $k\ge1$ and $u=(a_1/k,\ldots,a_n/k)$. We than say that $u$ is a \newword{rational point} whose \newword{denominator} is $k$, and write $k=\den(u)$.
A map $F:U\to\Rbb^n$, where $U$ is an open subset of $\Rbb^n$, 
\newword{preserves many denominators} (respectively, \newword{all denominators}) if it induces a bijection between $\set{u\in U\cap\Qbb^n:\den(u)=k}$ and
$\set{v\in F[U]\cap\Qbb^n:\den(v)=k}$ for infinitely many $k$'s (respectively, all $k$'s).
In~\S\ref{ref8} we will give various examples of denominator-preserving maps.
The map $F$ \newword{preserves the Lebesgue measure} $\lambda$ if 
both $F$ and $F^{-1}$ send $\lambda$-measurable sets to $\lambda$-measurable sets, with preservation of the measure. We will prove the following theorem.

\begin{theorem}
Let\labell{ref1} $F:U\to\Rbb^n$ be a continuous injective map that preserves many denominators. Then $F$ preserves the Lebesgue measure.
\end{theorem}

Let $g(k)$ be the number of rational points of denominator $k$ in the half-open cube $(0,1]^n$. 
If $f$ is a Riemann-integrable complex-valued function defined on a subset of $\Rbb^n$ then, by definition, $f$ is bounded and has compact support (the latter being the closure of $\set{u:f(u)\not=0}$). We then consider $f$ as defined on all of $\Rbb^n$, by setting
$f(u)=0$ for $u\notin\dom(f)$. With this understanding, 
Theorem~\ref{ref1} is a consequence of the following fact.

\begin{theorem}
For\labell{ref11} every Riemann-integrable function $f$ we have
\begin{equation}\tag{$\ast$}
\int_{\Rbb^n}f\,d\bar x=\lim_{k\to\infty}\frac{1}{g(k)}\sum_{\den(u)=k}f(u).
\end{equation}
\end{theorem}

Theorem~\ref{ref11} has another corollary, which is interesting in its own right since it extends to an $n$-dimensional setting the old finding that the Farey enumeration of all rational numbers in $\oi$ if uniformly distributed (see \cite{niederreiter73}, \cite[p.~136]{kuipersnie74}, \cite{kessebohmerstr10} and references therein).

\begin{theorem}
Let\labell{ref2} $u_1,u_2,u_3,\ldots$ be an enumeration without repetitions of all rational points in the half-open cube $(0,1]^n$. Assume that $r\le s$ implies $\den(u_r)\le\den(u_s)$. Then the sequence $(u_r)$ is uniformly distributed.
The same statement holds if $(0,1]^n$ is replaced by the closed cube $\oi^n$.
\end{theorem}

In Theorem~\ref{ref1} we are not assuming any regularity for the map $F$. In \S\ref{ref10} we will show that if $F$ is differentiable at a point $u$, then the Jacobian matrix w.r.t.~the standard basis is in $\GL_n\Zbb$ (this does not force $F$ to be affine in a neighborhood of $u$). As a corollary we get a version of Theorem~\ref{ref1} for bilipschitz maps, namely Corollary~\ref{ref12}, whose proof does not depend on Theorem~\ref{ref11}.

\section{Proof of Theorem~\ref{ref11}}

The function $g(k)$ counting the number of points of denominator $k$ in $(0,1]^n$ is Jordan's generalized
totient~\cite[p.~11]{murty08}, which reduces to Euler's totient for $n=1$.

\begin{lemma}
Denoting\labell{ref3} the M\"obius function by $\mu$, we have
$$
g(k)=\sum_{d|k}\mobiusfrac{k}{d}d^n.
$$
\end{lemma}
\begin{proof}
The number of rational points in $(0,1]^n$ that can be written in the form $(a_1/k,\ldots,a_n/k)$, with $\vect a1n,k$ not necessarily relatively prime, is $k^n$. Since the set of such points is the disjoint union of the sets of points having denominator $d|k$, we also have $k^n=\sum_{d|k}g(d)$, and we apply the M\"obius inversion formula.
\end{proof}

Given $u=(\vect t1n),v=(\vect l1n)\in\Rbb^n$, we write $u\le v$ (respectively, $u<v$) if $t_i\le l_i$ (respectively, $t_i < l_i$) for every $i$.
We denote the interval $\{x\in\Rbb^n:u<x\le v\}$ by $(u,v]$, and we first establish the identity ($\ast$) in the statement of Theorem~\ref{ref11} for all characteristic functions $f=\one_{(u,v]}$.
By partitioning $(u,v]$ in finitely many subintervals of the form $(u,v]\cap\bigl(w+(0,1]^n\bigr)$, with $w\in\Zbb^n$, we  assume without loss of generality that $u,v\in(0,1]^n$. Also, using the inclusion-exclusion principle it is easy to see that it suffices to establish ($\ast$) for intervals of the form $(0,v]$. Let then $v=(\vect l1n)\in(0,1]^n$ and $L=\prod l_i=\int\one_{(0,v]}\,d\bar x$.
Also, let $h(k)$ be the number of points of denominator $k$ in $(0,v]$; we have to prove that $\lim_{k\to\infty}h(k)/g(k)=L$.

The number of rational points in $(0,v]$ of the form $(a_1/k,\ldots,a_n/k)$, with $\vect ain,k$ not necessarily relatively prime, is $\floor{kl_1}\cdots\floor{kl_n}=\sum_{d|k}h(d)$. By M\"obius inversion we get
\begin{align*}
h(k)
&=\sum_{d|k}\mobiusfrac{k}{d}\floor{dl_1}\cdots\floor{dl_n} \\
&=\sum_{d|k}\mobiusfrac{k}{d}\big(dl_1-\set{dl_1}\big)
\cdots\big(dl_n-\set{dl_n}\big),
\end{align*}
where $\set{dl_i}=dl_i-\floor{dl_i}$ is the fractional part of $dl_i$. Now
\begin{equation*}
\big(dl_1-\set{dl_1}\big)
\cdots\big(dl_n-\set{dl_n}\big)=
d^nL+\sum_{\emptyset\not= J\subseteq\set{1,\ldots,n}}(-1)^{\abs{J}}
d^{n-\abs{J}}\prod_{\substack{j\in J\\ i\notin J}}l_i\set{dl_j},
\end{equation*}
and hence
$$
h(k)=Lg(k)+\sum_{d|k}\mobiusfrac{k}{d}\Biggl[\sum_{\emptyset\not= J}
(-1)^{\abs{J}}d^{n-\abs{J}}\prod_{\substack{j\in J\\ i\notin J}}
l_i\set{dl_j}\Biggr].
$$
Dividing by $g(k)$ and applying the triangle inequality we get
\begin{multline*}
\biggl\lvert\frac{h(k)}{g(k)}-L\biggr\rvert \le 
\frac{1}{g(k)}\sum_{d|k}\biggl\lvert\mobiusfrac{k}{d}\biggr\rvert
\biggl(\sum_{\emptyset\not= J}d^{n-\abs{J}}\biggr)\\
=\frac{1}{g(k)}\sum_{d|k}\biggl\lvert\mobiusfrac{k}{d}\biggr\rvert
\biggl(\sum_{t=1}^n \binom{n}{t} d^{n-t}\biggr)
\le
\frac{M}{g(k)}\sum_{d|k}\biggl\lvert\mobiusfrac{k}{d}\biggr\rvert
d^{n-1},
\end{multline*}
for some $M>0$, depending on $n$ only.

Hence it suffices to show that
$$
m(k)=\frac{\displaystyle{\sum_{d|k}
\biggl\lvert\mobiusfrac{k}{d}\biggr\rvert d^{n-1}}}
{\displaystyle{\sum_{d|k}\mobiusfrac{k}{d}d^n}}
$$
tends to $0$ as $k$ tends to infinity. Since $m$ is multiplicative, we can check this for $k$ assuming prime-power values~\cite[Theorem~316]{hardywri85}. We have
\begin{align*}
m(p^e)&=\frac{(p^e)^{n-1}+(p^{e-1})^{n-1}}{(p^e)^n-(p^{e-1})^n}\\
&=\frac{p^{en-e}+p^{en+1-e-n}}{p^{en}-p^{en-n}}\cdot
\frac{p^{n+e-en}}{p^{n+e-en}}\\
&=\frac{p^n+p}{p^{n+e}-p^e}\\
&=\frac{p(p^{n-1}+1)}{(p-1)(p^{n-1}+p^{n-2}+\cdots+1)}\cdot
\frac{1}{p^e}\\
&\le 4\cdot\frac{1}{p^e},
\end{align*}
that tends to $0$ as $p^e$ tends to infinity.

This establishes~($\ast$) for the characteristic functions of intervals, and it is clear that~($\ast$) must then hold for all finite linear combinations of such characteristic functions.
The extension of~($\ast$) to all Riemann-integrable functions is now straightforward (see, e.g., the proof of~\cite[Theorem~1.1.1]{kuipersnie74}).

\section{Proof of Theorem~\ref{ref1}}

Let $F,U$ be as in the statement of Theorem~\ref{ref1}.
As a consequence of Brouwer's invariance of domain theorem~\cite[p.~217]{massey91}, $V=F[U]$ is open in $\Rbb^n$ and $F:U\to V$ is a homeomorphism. 
By our assumptions on $F$, there exists a sequence $k_1<k_2<k_3<\cdots$ such that $F$ induces a bijection between the points of denominator $k_i$ in $U$ and those in $V$, for every $i$. By Theorem~\ref{ref11}, for every Riemann-integrable function $f$ whose support is contained in $V$ we have
\begin{align*}
\int_{\Rbb^n}f\,d\bar x&=
\lim_{i\to\infty}\frac{1}{g(k_i)}
\sum_{\substack{v\in V\\ \den(v)=k_i}}f(v)\\
&=
\lim_{i\to\infty}\frac{1}{g(k_i)}
\sum_{\substack{u\in U\\ \den(u)=k_i}}(f\circ F)(u)
=
\int_{\Rbb^n}f\circ F\,d\bar x.
\end{align*}
Let $W$ be an open subset of $V$. By the construction of the Lebesgue measure $\lambda$~\cite[proof of Theorem~2.14]{rudin87}, $\lambda(W)$ is the least upper bound of the values $\int f\,d\bar x$, where $f$ ranges in the set of all continuous, $\oi$-valued functions supported in $W$. Since this set of functions is in 1-1 correspondence ---via postcomposition with $F$--- with the analogous set of functions supported in $F^{-1}W$, the identity displayed above yields 
$\lambda(F^{-1}W)=\lambda(W)$.
By~\cite[Theorem~2.14(c)]{rudin87}, 
$\lambda(F^{-1}A)=\lambda(A)$ for every $A\subseteq V$ such that both $A$ and $F^{-1}A$ are $\lambda$-measurable, in particular for every Borel set. By~\cite[Theorem~2.17]{rudin87}, 
$A\subseteq\Rbb^n$ is $\lambda$-measurable iff there exist $B\in F_\sigma$, $C\in G_\delta$ with $B\subseteq A\subseteq C$ and $\lambda(C\setminus B)=0$; if this happens, $\lambda(A)=\lambda(C)$. Let then $A\subseteq V$ be $\lambda$-measurable, $B$ and $C$ be as above with $C\subseteq V$. Then $F^{-1}B\subseteq F^{-1}A\subseteq F^{-1}C$, with $F^{-1}B\in F_\sigma$ and $F^{-1}C\in G_\delta$. Since $C\setminus B$ is Borel, we have by the above $\lambda(F^{-1}C\setminus F^{-1}B)=
\lambda\bigl(F^{-1}(C\setminus B)\bigr)=\lambda(C\setminus B)=0$; hence $F^{-1}A$ is $\lambda$-measurable and $\lambda(F^{-1}A)=
\lambda(F^{-1}C)=\lambda(C)=\lambda(A)$.
The same argument applies to $F^{-1}$, so the $\lambda$-measurability is preserved in both directions.

\section{Proof of Theorem~\ref{ref2}}

Let $X$ denote either the half-open cube $(0,1]^n$ of the closed cube $\oi^n$. We recall that a sequence $(u_r)_{r\in\Nbb}$ in $X$ is \newword{uniformly distributed} if for every Riemann-integrable function $f:X\to\Cbb$ we have
$$
\lim_{s\to\infty}
\frac{1}{s}\sum_{r=1}^s f(u_r)=
\int_Xf\,d\bar x.
$$

We let $G(k)$ denote the number of rational points of denominator $k$ in $\oi^n$; as in Lemma~\ref{ref3} we have
$$
G(k)=\sum_{d|k}\mobiusfrac{k}{d}(d+1)^n.
$$
We need a few facts about the growth rate of $g(k)$ and $G(k)$; the following estimates are known~\cite[Exercises~1.3.4, 1.3.5, 1.5.3]{murty08}:
\begin{itemize}
\item[(a)] for $n=1$, there exist positive constants $C_1,C_2$ such that
$$
C_1\frac{k}{1+\log k}\le g(k)\le C_2k;
$$
\item[(b)] for $n\ge2$, there exist positive constants $C_1,C_2$ (depending on $n$) such that
$$
C_1k^n\le g(k)\le C_2k^n.
$$
\end{itemize}
Let $t(k)=g(1)+g(2)+\cdots+g(k)$ and $T(k)=G(1)+G(2)+\cdots+G(k)$.

\begin{lemma}
We\labell{ref4} have
$$
\lim_{k\to\infty}\frac{g(k+1)}{t(k)}
=0,
$$
and analogously for $G,T$.
Also, $\lim_{k\to\infty}g(k)/G(k)=1$.
\end{lemma}
\begin{proof}
It is known~\cite[p.~155]{comtet74} that
$$
1^n+2^n+\cdots+k^n=\frac{1}{n+1}k^{n+1}+O(k^n).
$$
The claim for $g$ follows then easily from (a) and (b) above.
As we now need specify the dimension $n$, we will write $g_n$ for $g$, and analogously for $G$, $t$, $T$, till the end of the proof.
Since the vertex $(0,0,\ldots,0)$ is contained in $n$ maximal faces of the unit cube $\oi^n$, we have $G_n(k)\le g_n(k)+nG_{n-1}(k)$ (we set $G_0$ to be identically $1$). Therefore
$$
\frac{G_n(k+1)}{T_n(k)}
\le
\frac{g_n(k+1)+nG_{n-1}(k+1)}{t_n(k)},
$$
and the latter tends to $0$ as $k$ tends to infinity.

For the last statement we observe that $G_n(k)\le g_n(k)+nG_{n-1}(k)$ implies
$$
1-n\frac{G_{n-1}(k)}{G_n(k)}\le
\frac{g_n(k)}{G_n(k)}<1.
$$
Also, we have the trivial bound
$$
\frac{G_{n-1}(k)}{G_n(k)}\le
\frac{(k+1)^{n-1}}{g_n(k)},
$$
whence $\lim_{k\to\infty}G_{n-1}(k)/G_n(k)=0$ by (a) and (b).
\end{proof}

Let $(\alpha_r)_{r\in\Nbb}$ be any element of $\ell_\infty(\Cbb)$, 
and let $\beta\in\Cbb$. Three summation methods are relevant for us:
\begin{itemize}
\item[(1)] \newword{block convergence}
$$
\lim_{k\to\infty}\frac{1}{g(k)}\sum_{r=t(k-1)+1}^{t(k)}\alpha_r=\beta;
$$
\item[(2)] \newword{Ces\`aro convergence}
$$
\lim_{s\to\infty}\frac{1}{s}\sum_{r=1}^s\alpha_r=\beta;
$$
\item[(3)] \newword{blockwise Ces\`aro convergence}
$$
\lim_{k\to\infty}\frac{1}{t(k)}\sum_{r=1}^{t(k)}\alpha_r=\beta.
$$
\end{itemize}
Clearly (2)\To(3), and (1)\To(3) is an easy consequence of a theorem by Cauchy~\cite[Lemma~2.4.1]{kuipersnie74}. Since, as proved in Lemma~\ref{ref4}, the ratio $g(k+1)/t(k)$ tends to $0$,
we have (3)\To(2) by~\cite[Lemma~2.4.1]{kuipersnie74}. Returning to the proof of Theorem~\ref{ref2}, assume first $X=(0,1]^n$ and let $(u_r)$ be a sequence as in the statement of Theorem~\ref{ref2}. Fix a Riemann-integrable function $f$ on $X$, let $\alpha_r=f(u_r)$ and $\beta=\int_X f\,d\bar x$. The points $\set{u_r:t(k-1)<r\le t(k)}$ are precisely the points of denominator $k$ in $X$, in some order, so (1) holds by Theorem~\ref{ref11}. Therefore (2) holds as well and Theorem~\ref{ref2} is proved for the half-open cube.

Take now $X=\oi^n$, and let $(u_r)$, $f$, $(\alpha_r)$, $\beta$ be defined as above with the obvious modifications. By Lemma~\ref{ref4} $G(k+1)/T(k)\to0$, so the above discussion holds verbatim and we only need to prove
$$
\lim_{k\to\infty}\frac{1}{G(k)}\sum_{r=T(k-1)+1}^{T(k)}\alpha_r=\beta.
$$
Without loss of generality the first $g(k)$ points in $u_{T(k-1)+1},\ldots,u_{T(k)}$ are in $(0,1]^n$, and the remaining ones in $\oi^n\setminus(0,1]^n$. We thus get
$$
\frac{1}{G(k)}\sum_{r=T(k-1)+1}^{T(k)}\alpha_r
=\frac{g(k)}{G(k)}\cdot
\frac{1}{g(k)}\sum_{r=T(k-1)+1}^{T(k-1)+g(k)}\alpha_r+
\frac{1}{G(k)}\sum_{r=T(k-1)+g(k)+1}^{T(k)}\alpha_r.
$$
For $k$ tending to infinity $g(k)/G(k)$ tends to $1$ by Lemma~\ref{ref4}, hence the first summand to the right-hand side tends to $\beta$ as proved above. On the other hand the second summand tends to $0$, because its absolute value is dominated by $\norm{f}_\infty\bigl(G(k)-g(k)\bigr)/G(k)$. This concludes the proof of Theorem~\ref{ref2}.

\begin{remark}
In general block convergence is strictly stronger than blockwise Ces\`aro convergence; examples are easily constructed. It is precisely the stronger property (1) allowing us to ask for preservation of many denominators ---rather than all denominators--- in Theorem~\ref{ref1}. This turns out to be handy, e.g., in Lemma~\ref{ref6} below. Equidistribution results for rational points in algebraic varieties are usually formulated in terms of blockwise
Ces\`aro convergence. Points are sorted according to their height, and the resulting measure is the finite Tamagawa measure; see the survey~\cite{peyre03}.
\end{remark}

\section{Translations and differentiability}\labell{ref10}

Translations by rational vectors preserve many denominators.

\begin{lemma}
Let\labell{ref6} $v=d^{-1}(\vect a1n)\in\Qbb^n$ have denominator $d$, let $U\subseteq\Rbb^n$ be open and let $T_v:U\to\Rbb^n$ be the translation by $v$. Then:
\begin{itemize}
\item[(i)] $T_v$ preserves every denominator $k$ such that $d\rad(d)|k$ ($\rad(d)$ being the product of all prime factors of $d$, each taken with exponent $1$);
\item[(ii)] $T_v$ preserves all denominators iff $v\in\Zbb^n$.
\end{itemize}
\end{lemma}
\begin{proof}
Denote by $\op:\Qbb\to\Zbb\cup\set{+\infty}$ the $p$-adic valuation (i.e., $\op(l)$ is the exponent at which the prime $p$ appears in the unique factorization of $l\in\Qbb\setminus\set{0}$).
Let $u=k^{-1}(\vect b1n)\in U$ have denominator $k$, the latter being a multiple of $d\rad(d)$. We will show that $k|\den(T_v(u))$. 
Since 
$\den(v)=\den(-v)$, it will follow that
both $T_v$ and $T_v^{-1}$ map points whose denominator is a multiple $k$ of $d\rad(d)$ to points whose denominator is a multiple of $k$, so that both $T_v$ and $T_v^{-1}$ must preserve such denominators. Let then $p|k$, with the intent of proving $\op(k)\le\op\bigl(\den(T_v(u))\bigr)$.

For at least one index $i$ we have $p\nmid b_i$, and for that index
$\op(b_i/k)=\op(1/k)<\op(1/d)\le\op(a_i/d)$; the middle inequality is clear if $p\nmid d$, and follows from $d\rad(d)|k$ otherwise. Therefore $\op(a_i/d+b_i/k)=\min\{\op(a_i/d),\op(b_i/k)\}=-\op(k)$,
and hence $0<\op(k)=-\op(a_i/d+b_i/k)\le\op\bigl(\den(T_v(u))\bigr)$.
We thus proved (i); for the nontrivial direction of (ii), assume $v\notin\Zbb^n$. Then some prime $p$ divides $\den(v)$, and the open set $U$ surely contains a point $u$ with $p\nmid\den(u)$. Therefore $p|\den(T_v(u))$, and $T_v$ does not preserve all denominators.
\end{proof}

For the remaining of this paper we will be concerned only with maps that preserve all denominators. Under differentiability assumptions denominator-preserving maps have a rigid structure, as explained in the following theorem.

\begin{theorem}
Let\labell{ref7} $F:U\to\Rbb^n$ be a continuous injective map that preserves all denominators. Assume that $F$ is differentiable at $u$ with Jacobian matrix A w.r.t.~the standard basis of $\Rbb^n$. Then $A\in\GL_n\Zbb$. In particular, if $F$ is continuously differentiable on the open connected region $V\subseteq U$, then $F$ has the form
$F(v)=Av+w$ on $V$, for some fixed $A\in\GL_n\Zbb$ and $w\in\Zbb^n$.
\end{theorem}

In~\S\ref{ref8} we will give an example of a denominator-preserving homeomorphism of $\Rbb^2$ that is differentiable at the origin, but fails to be affine in any neighborhood of the origin. We will also construct a nowhere differentiable denominator-preserving homeomorphism of $\Rbb^2$.

In order to prove Theorem~\ref{ref7} we need a few preliminaries.
Let $\sigma=(\vect v0n)$ denote an ordered $(n+1)$-tuple of vectors in general position in $U$. We say that a sequence $\sigma^1,\sigma^2,\sigma^3,\ldots$ of such tuples, with $\sigma^t=(\vect{v^t}0n)$, \newword{converges to the point $u$} if $u$ is in the convex hull of every $\sigma^t$ and $v^t_i$ converges to $u$ for every $i$. Every $\sigma$ determines a linear map $L_\sigma:\Rbb^n\to\Rbb^n$ by $L_\sigma(v_i-v_j)=F(v_i)-F(v_j)$. 
The following calculus lemma might be known, but we have not been able to find a reference.

\begin{lemma}
Let\labell{ref9} $F$ be any map from an open subset $U$ of $\Rbb^n$ to $\Rbb^n$. Assume that $F$ is differentiable at $u\in U$ with differential $L\in\Lin(\Rbb^n,\Rbb^n)$, and let $(\sigma^t)$ be a sequence converging to $u$ as above, with $v^t_i\in U$ for every $t$ and $i$. Then $L_{\sigma^t}$ converges to $L$ in the operator norm.
\end{lemma}
\begin{proof}
Fix temporarily $t$, let $\sigma=\sigma^t=(\vect v0n)$ and $u=\sum_i\alpha_iv_i$, with $\sum_i\alpha_i=1$ and $\vect\alpha0n\ge0$. Denote by $\sigma_i$ the tuple obtained from $\sigma$ by replacing $v_i$ with $u$. We claim that $L_\sigma=\sum_i\alpha_iL_{\sigma_i}$.
Note that if $u$ belongs to the affine subspace spanned by $\vect v0{{i-1}},\vect v{{i+1}}n$, then $L_{\sigma_i}$ is undefined, but this happens precisely when $\alpha_i=0$, so the above identity still makes sense. 
Note also that if $\vect\beta0n\in\Rbb$ are such that $\sum_i\beta_i=0$, then $L_\sigma(\sum_i\beta_iv_i)=\sum_i\beta_iF(v_i)$; this is easily proved by induction on the number of indices $i$ such that $\beta_i\not=0$.

Since both sides of the claimed equality are linear maps, it suffices to show that they agree on all differences $v_j-v_k$, i.e., that $F(v_j)-F(v_k)=\sum_{\alpha_i\not=0}L_{\sigma_i}\bigl(\alpha_i(v_j-v_k)\bigr)$ for every $j\not=k$. We fix then $j\not=k$, and compute $L_{\sigma_i}\bigl(\alpha_i(v_j-v_k)\bigr)$ under the assumption $\alpha_i\not=0$. We obtain:
\begin{itemize}
\item[(i)] if $i\not=j,k$, then $L_{\sigma_i}\bigl(\alpha_i(v_j-v_k)\bigr)=\alpha_i\bigl(F(v_j)-F(v_k)\bigr)$;
\item[(ii)] if $i=j$, then $L_{\sigma_j}\bigl(\alpha_j(v_j-v_k)\bigr)=
F(u)-(\alpha_k+\alpha_j)F(v_k)-\sum_{l\not=j,k}\alpha_lF(v_l)$;
\item[(iii)] if $i=k$, then $L_{\sigma_k}\bigl(\alpha_k(v_j-v_k)\bigr)=
-F(u)+(\alpha_j+\alpha_k)F(v_j)+\sum_{l\not=j,k}\alpha_lF(v_l)$.
\end{itemize}
Indeed, (ii) is true since $u=\sum_l\alpha_lv_l$ implies $\alpha_j(v_j-v_k)=u-(\alpha_k+\alpha_j)v_k-\sum_{l\not=j,k}\alpha_lv_l$ and $1-(\alpha_k+\alpha_j)-\sum_{l\not=j,k}\alpha_l=0$, and (iii) is analogous. Summing up everything we obtain
\begin{multline*}
\sum_{\alpha_i\not=0}L_{\sigma_i}\bigl(\alpha_i(v_j-v_k)\bigr)=
(\alpha_j+\alpha_k)F(v_j)-(\alpha_k+\alpha_j)F(v_k)\\
+\sum_{i\not=j,k}\alpha_i\bigl(F(v_j)-F(v_k)\bigr)
=F(v_j)-F(v_k),
\end{multline*}
which settles our claim.

Let now $\varepsilon>0$. Since $F$ is differentiable at $u$, there exists an index $t'$ such that for every $i=0,\ldots,n$ and every $t\ge t'$ the linear map $L_{\sigma^t_i}$ is either undefined, or is in the open ball of center $L$ and radius $\varepsilon$ in $\Lin(\Rbb^n,\Rbb^n)$. Since such a ball is convex, it contains $L_{\sigma^t}$ by our claim above.
\end{proof}

We can now prove Theorem~\ref{ref7}. Using the M\"onkemeyer-Selmer multidimensional continued fractions algorithm (or any other topologically convergent procedure, see~\cite{panti08} for the
M\"onkemeyer-Selmer algorithm or~\cite{schweiger00} for a complete panorama) it is easy to construct a sequence $(\sigma^t)$ converging
to $u$ such that each tuple $\sigma^t=(\vect{v^t}0n)$ is \newword{unimodular}. 
By this we mean that every $v_i^t$ is in $\Qbb^n$ and the $(n+1)\times(n+1)$ integer matrix
$$
S_t=\begin{pmatrix}
a^t_{1,0} & \cdots & a^t_{1,n} \\
\vdots & \cdots & \vdots \\
a^t_{n,0} & \cdots & a^t_{n,n} \\
d^t_0 & \cdots & d^t_n
\end{pmatrix}
$$
whose columns are the projective coordinates of $\vect {v^t}0n$, with $d^t_i=\den(v^t_i)$, is in $\GL_{n+1}\Zbb$. Let analogously
$$
W_t=\begin{pmatrix}
b^t_{1,0} & \cdots & b^t_{1,n} \\
\vdots & \cdots & \vdots \\
b^t_{n,0} & \cdots & b^t_{n,n} \\
d^t_0 & \cdots & d^t_n
\end{pmatrix}
$$
be the integer matrix whose columns are the projective coordinates of $F(v^t_0),\ldots,F(v^t_n)$. Then the $n\times n$ matrix $A_t$ that expresses $L_{\sigma^t}$ w.r.t.~the standard basis of $\Rbb^n$ is the upper left minor of the $(n+1)\times(n+1)$ matrix $B_t$ defined by the identity
\begin{multline*}
B_t\begin{pmatrix}
v_0-v_n & \cdots & v_{n-1}-v_n & v_n \\
0 & \cdots & 0 & 1
\end{pmatrix}\\
=
\begin{pmatrix}
F(v_0)-F(v_n) & \cdots & F(v_{n-1})-F(v_n) & F(v_n) \\
0 & \cdots & 0 & 1
\end{pmatrix};
\end{multline*}
here the vectors on the top rows are $n$-rows column vectors.
Multiplying to the right both sides of the above identity first by an appropriate elementary matrix, and then by the diagonal matrix whose diagonal entries are $d^t_0,\ldots,d^t_n$, we get the identity $B_tS_t=W_t$, which implies that $B_t=W_tS_t^{-1}$ has integer entries. Therefore $A_t$ has integer entries. By Lemma~\ref{ref9} $A_t$ converges to $A$ for $t$ going to infinity, and since $\Mat_{n\times n}\Zbb$ is discrete in $\Mat_{n\times n}\Rbb$ we conclude that $A$ has integer entries. By Theorem~\ref{ref1} $F$ preserves the Lebesgue measure, so $A$ is in $\GL_n\Zbb$.

If $F$ is continuously differentiable on the open connected region $V$, then clearly the Jacobian matrix $A$ must be constant, since $\GL_n\Zbb$ is discrete.
The map $F(v)-Av$ has then null differential on $V$, so it is constant~\cite[(8.6.1)]{dieudonne69}; therefore $F(v)=Av+w$ on $V$, for some fixed column vector $w\in\Rbb^n$. As the translation by $A^{-1}w$ equals the composite map $A^{-1}\circ F:V\to\Rbb^n$, it preserves all denominators. By Lemma~\ref{ref6}(ii) $A^{-1}w\in\Zbb^n$, and hence $w\in\Zbb^n$; this concludes the proof of Theorem~\ref{ref7}.

\smallskip

Theorem~\ref{ref7} yields a weaker version of Theorem~\ref{ref1}, whose proof is independent of Theorem~\ref{ref11}.

\begin{corollary}
Let\labell{ref12} $F$ be a bilipschitz homeomorphism of open subsets of $\Rbb^n$ that preserves all denominators. Then $F$ preserves the Lebesgue measure.
\end{corollary}
\begin{proof}
By Rademacher's theorem~\cite[Theorem~2, p.~81]{evansgariepy92}
both $F$ and $F^{-1}$ are differentiable $\lambda$-a.e.. Moreover, by~\cite[Lemma~7.25]{rudin87} the $F^{-1}$-image of the set of nondifferentiability points of $F^{-1}$ is a Lebesgue nullset. Therefore, for $\lambda$-a.e.~$u\in\dom(F)$, $F$ is differentiable at $u$ and $F^{-1}$ is differentiable at $F(u)$.
Let $A$ be the Jacobian matrix of $F$ at $u$ and $B$ that of $F^{-1}$ at $F(u)$. By the proof of Theorem~\ref{ref7} both $A$ and $B$ have integer entries; since their product is the identity matrix, both of them are in $\GL_n\Zbb$. Note that in the proof of Theorem~\ref{ref7} we employed Theorem~\ref{ref11} only in showing $\abs{\det(A)}=1$; this is automatic here due to our stronger hypotheses. The conclusion now follows from the area formula~\cite[Theorem~1, p.~96]{evansgariepy92}.
\end{proof}

\section{Examples of denominator-preserving maps}\labell{ref8}

The gingerbreadman map $F(x,y)=(1-y+\abs{x},x)$ is a well known example of a denominator-preserving area-preserving homeomorphism of $\Rbb^2$ possessing interesting dynamical properties~\cite{devaney84}. It has a unique elliptic fixed point at $(1,1)$, which is surrounded by infinitely many polygonal annuli on which the dynamics is hyperbolic and chaotic in regions of positive measure.

Denominator-preserving volume-preserving homeomorphisms of the unit even-dimensional cube are constructed in~\cite{pantibernoulli}. These are linked twist maps, fixing the boundary of the cube, ergodic and uniformly hyperbolic throughout the whole domain.

We close this paper presenting two examples of denominator-preserving area-preserving homeomorphisms of $\Rbb^2$ which are related to (non)differentiability issues. Both of them are defined with the help of an auxiliary function $f:\Rbb\to\Rbb$. Given such an $f$ we define $F:\Rbb^2\to\Rbb^2$ by $F(x,y)=(x,y+f(x))$; $F$ is then an area-preserving bijection, is a homeomorphism iff $f$ is continuous, and is differentiable at every point of the line $\set{x=\alpha}$ iff $f$ is differentiable at $\alpha$. Moreover, $F$ preserves all denominators, provided that $f(l)\in\bigl(\den(l)\bigr)^{-1}\Zbb$ for every $l\in\Qbb$. Indeed, say that $u=d^{-1}(a,b)$ has denominator $d$. Then $l=a/d$ has denominator $e=d/(a,d)$ and $f(l)=(a,d)c/d$ for some $c\in\Zbb$. Hence $F(u)=d^{-1}(a,b+(a,d)c)$ has denominator $d$, since $(a,b,d)=1$ implies $(a,b+(a,d)c,d)=1$. The same argument applies to $F^{-1}$ (which is induced by $-f$), and hence both $F$ and $F^{-1}$ map points of a certain denominator to points of the same denominator, so both of them preserve all denominators.

We construct our first example by defining a sequence $p_1,p_2,\ldots$ of prime numbers as follows: $p_1$ is $2$, and $p_{t+1}$ is the least prime strictly greater than $(1+1/t)p_t$. The sequence $t/p_t$ is strictly decreasing (the initial terms are $1/2$, $2/5$, $3/11$, $4/17$, $5/23$, $\ldots$), and converges to $0$ because it is dominated by the sequence $t/(\text{the $t$-th prime number})$, that converges to $0$ by the Prime Number Theorem. For each $t$ choose two rational numbers $a_t/b_t$, $c_t/d_t$ such that
$$
\frac{c_{t+1}}{d_{t+1}} \le \frac{a_t}{b_t} <
\frac{t}{p_t} < \frac{c_t}{d_t},
$$
and each of the two intervals $[a_t/b_t,t/p_t]$, $[t/p_t,c_t/d_t]$ is unimodular, i.e.,
$$
\begin{vmatrix}
t & a_t\\
p_t & b_t
\end{vmatrix}
=
\begin{vmatrix}
c_t & t\\
d_t & p_t
\end{vmatrix}
=1.
$$
This is easily accomplished either by using continued fractions or by using Farey partitions~\cite[Chapter~III]{hardywri85}. Define now $f$ by
$$
f(x)=
\begin{cases}
b_tx-a_t, & \text{if $a_t/b_t\le x \le t/p_t$;}\\
-d_tx+c_t, & \text{if $t/p_t\le x \le c_t/d_t$;}\\
0, & \text{otherwise.}
\end{cases}
$$
Then $f$ is continuous and satisfies $f(l)\in\bigl(\den(l)\bigr)^{-1}\Zbb$ for every $l\in\Qbb$. Moreover, it is differentiable at each 
$\alpha\in\Rbb\setminus\bigcup_{t\ge1}\set{a_t/b_t,t/p_t,c_t/d_t}$.
Indeed, the only problematic point is $\alpha=0$. However, the ratio $f(x)/x$ is $0$ for $x$ outside the intervals $[a_t/b_t,c_t/d_t]$, and has value
$$
M_t(x)=
\frac{\min\set{b_tx-a_t,-d_tx+c_t}}{x},
$$
on $[a_t/b_t,c_t/d_t]$. It is easily checked that $M_t$ attains its maximum value $(1/p_t)/(t/p_t)=1/t$ at $x=t/p_t$; hence $f$ is differentiable at $0$ with derivative~$0$. The map $F$ induced by $f$ as above is then a homeomorphism of $\Rbb^2$ which preserves all denominators, is differentiable at the origin with differential the identity map ---in accordance with Theorem~\ref{ref7}--- but is not affine in any neighborhood of the origin.

In our second example we construct a nowhere differentiable denominator-preserving homeomorphism of $\Rbb^2$.
Recall the construction of the Stern-Brocot sequence on the real unit interval~\cite{kessebohmerstr10}; the only interval \newword{belonging} to stage $0$ is $[0/1,1/1]$. At stage $t+1$ each of the $2^t$ intervals $[a/b,c/d]$ belonging to stage $t$ is split into two intervals $[a/b,(a+c)/(b+d)]$ and $[(a+c)/(b+d),c/d]$; the point $(a+c)/(b+d)$ is the \newword{Farey mediant} of $a/b$ and $c/d$. For $t=1,2,3,\ldots$, define $g_t:\oi\to\Rbb$ as follows:
\begin{itemize}
\item $g_t=0$ at all endpoints of all intervals belonging to stage $t-1$;
\item $g_t=1/\den(u)$ at each Farey mediant $u$ appearing at stage $t$;
\item $g_t$ is affine linear on each interval belonging to stage $t$.
\end{itemize}
The functions $g_t$ are sawlike, $g_1\ge g_2\ge g_3\ge\cdots$ on $\oi$,
$\norm{g_t}_\infty=1/(t+1)$, each interval $[u,v]$ belonging to stage $t$ is unimodular, and $g_t$ has slope $\pm\den(u)\in\Zbb\setminus\set{0}$ on $[u,v]$, provided that the Farey mediant inserted at step $t$ is $v$; otherwise $g_t$ has slope $\pm\den(v)$ on $[u,v]$. Consider the alternating sum $\sum_{1\le t}(-1)^{t-1}g_t$; by the Leibniz test it converges uniformly to a continuous function $f$. Given a rational number $0\le l\le 1$, we have $f(l)\in\bigl(\den(l)\bigr)^{-1}\Zbb$; indeed $l$ appears as an endpoint at some stage $t$ of the procedure, and then $g_m(l)=0$ for every $m>t$. Hence $f(l)=\sum_{1\le m\le t}(-1)^{m-1}g_m(l)$, and each $g_m(l)$ is in $\bigl(\den(l)\bigr)^{-1}\Zbb$.

We extend $f$ to a period-$1$ function defined on all of $\Rbb$ in the obvious way, and we claim that $f$ is of Weierstrass type, i.e., is nowhere differentiable. Indeed, let $0\le\alpha\le1$, and let $I_1\supset I_2\supset I_3\supset\cdots$ be a chain of intervals with $I_t$ belonging to stage $t$ and $\bigcap_{t\ge1}I_t=\set{\alpha}$ (this chain is unique iff $\alpha\not=0,1$ is irrational, and is determined by the continued fraction expansion of $\alpha$). Let $I_t=[u_t,v_t]$, let $r_t$ be the slope of $g_t$ on $I_t$, and let
$s_t=\bigl(f(v_t)-f(u_t)\bigr)/(v_t-u_t)$. Since each $g_m$, for $1\le m\le t$, is linear on $I_t$, and $g_m(u_t)=g_m(v_t)=0$ for $m>t$, the number $s_t$ is precisely the slope of $f_t=\sum_{1\le m\le t}(-1)^{m-1}g_m$ on $I_t$, i.e., $s_t=\sum_{1\le m\le t}(-1)^{m-1}r_m$.
If $f$ were differentiable at $\alpha$ then, by Lemma~\ref{ref9}, $\lim_{t\to\infty}s_t$ would exist. But this is impossible, since $r_m\in\Zbb\setminus\set{0}$ for each $m$.
\begin{figure}[H]
\begin{center}
\includegraphics[height=7cm]{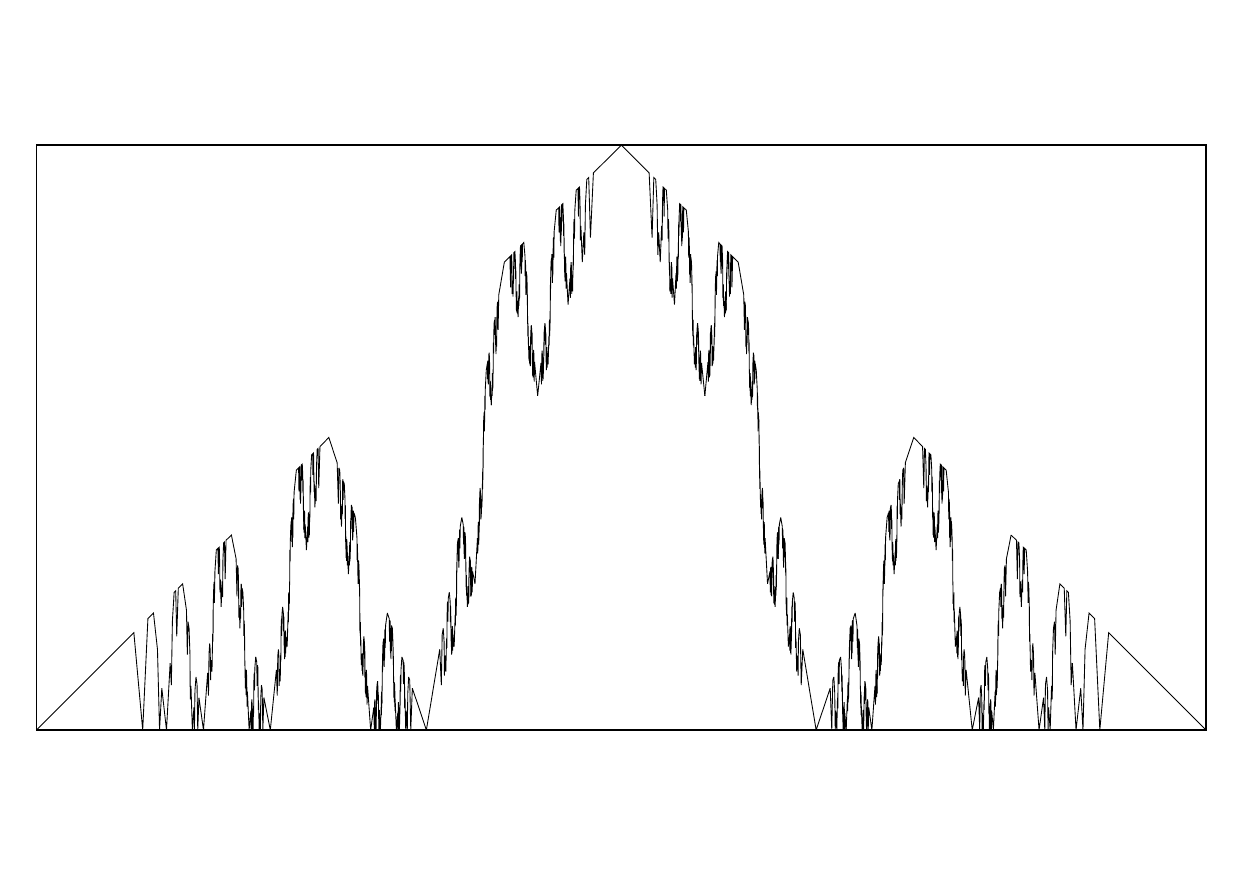}\\
\vspace{-0.8cm}
{Graph of $f_{11}$.}
\end{center}
\end{figure}


\begin{thebibliography}{10}

\bibitem{comtet74}
L.~Comtet.
\newblock {\em Advanced combinatorics}.
\newblock D. Reidel Publishing Co., Dordrecht, 1974.

\bibitem{devaney84}
R.~L. Devaney.
\newblock A piecewise linear model for the zones of instability of an
  area-preserving map.
\newblock {\em Phys. D}, 10(3):387--393, 1984.

\bibitem{dieudonne69}
J.~Dieudonn{\'e}.
\newblock {\em Foundations of modern analysis}.
\newblock Academic Press, New York, 1969.
\newblock Enlarged and corrected printing, Pure and Applied Mathematics, Vol.
  10-I.

\bibitem{evansgariepy92}
L.~C. Evans and R.~F. Gariepy.
\newblock {\em Measure theory and fine properties of functions}.
\newblock Studies in Advanced Mathematics. CRC Press, Boca Raton, FL, 1992.

\bibitem{hardywri85}
G.~H. Hardy and E.~M. Wright.
\newblock {\em An introduction to the theory of numbers}.
\newblock Oxford University Press, 5th edition, 1985.

\bibitem{kessebohmerstr10}
M.~Kesseb\"ohmer and B.~O. Stratmann.
\newblock A dichotomy between uniform distributions of the {S}tern-{B}rocot and
  the {F}arey sequences.
\newblock \texttt{http://arxiv.org/abs/1009.1823}, 2010.

\bibitem{kuipersnie74}
L.~Kuipers and H.~Niederreiter.
\newblock {\em Uniform distribution of sequences}.
\newblock Dover, 2006.
\newblock First published in 1974 by Wiley-Interscience.

\bibitem{massey91}
W.~S. Massey.
\newblock {\em A basic course in algebraic topology}, volume 127 of {\em
  Graduate Texts in Mathematics}.
\newblock Springer-Verlag, New York, 1991.

\bibitem{murty08}
M.~R. Murty.
\newblock {\em Problems in analytic number theory}, volume 206 of {\em Graduate
  Texts in Mathematics}.
\newblock Springer-Verlag, New York, 2001.

\bibitem{niederreiter73}
H.~Niederreiter.
\newblock The distribution of {F}arey points.
\newblock {\em Math. Ann.}, 201:341--345, 1973.

\bibitem{pantibernoulli}
G.~Panti.
\newblock Bernoulli automorphisms of finitely generated free {MV}-algebras.
\newblock {\em J. Pure Appl. Algebra}, 208(3):941--950, 2007.

\bibitem{panti08}
G.~Panti.
\newblock Multidimensional continued fractions and a {M}inkowski function.
\newblock {\em Monatshefte f\"ur Mathematik}, 154:247--264, 2008.

\bibitem{peyre03}
E.~Peyre.
\newblock Points de hauteur born\'ee, topologie ad\'elique et mesures de
  {T}amagawa.
\newblock {\em J. Th\'eor. Nombres Bordeaux}, 15(1):319--349, 2003.
\newblock Les XXII{\`e}mes Journ{\'e}es Arithmetiques (Lille, 2001).

\bibitem{rudin87}
W.~Rudin.
\newblock {\em Real and complex analysis}.
\newblock McGraw-Hill Book Co., New York, third edition, 1987.

\bibitem{schweiger00}
F.~Schweiger.
\newblock {\em Multidimensional continued fractions}.
\newblock Oxford Science Publications. Oxford University Press, Oxford, 2000.

\end{thebibliography}

\end{document}